\documentclass[
]{ceurart}

\usepackage{mathtools}
\usepackage{amsthm}
\usepackage{tikz}
\usetikzlibrary{arrows,fit,positioning}

\usepackage{forest}
\tikzset{
	treenode/.style = {align=center, inner sep=0pt, text centered},
	Ske/.style = {treenode, ellipse, double, draw=black,
		minimum width=6pt, thick},
	PIA/.style = {treenode, ellipse, black, draw=black,
		minimum width=6pt},
	Crit/.style = {treenode, rectangle, draw=black,
		minimum width=0.5em, minimum height=0.5em}
}
\usetikzlibrary{patterns}
\usepackage{bussproofs}
\EnableBpAbbreviations

\newtheorem{thm}{Theorem}[section]
\newtheorem{theorem}{Theorem}[section]
\newtheorem{corollary}[thm]{Corollary}
\newtheorem{lemma}[thm]{Lemma}
\newtheorem{prop}[thm]{Proposition}
\newtheorem{proposition}[thm]{Proposition}

\newtheorem{notation}[thm]{Notation}

\newtheorem{example}[thm]{Example}
\newtheorem{definition}[thm]{Definition}

\newcommand{\FH}{\hat{f}}

\newcommand{\GC}{\check{g}}

\newcommand{\FHS}{\hat{f}^{\,\sharp}}
\newcommand{\FCS}{\check{f}^{\,\sharp}}
\newcommand{\GHF}{\hat{g}^{\,\flat}}
\newcommand{\GCF}{\check{g}^{\,\flat}}

\newcommand{\aatop}{\ensuremath{\top}\xspace}
\newcommand{\abot}{\ensuremath{\bot}\xspace}
\newcommand{\aand}{\ensuremath{\wedge}\xspace}
\newcommand{\aor}{\ensuremath{\vee}\xspace}

\newcommand{\NEG}{\ensuremath{\:\check{\neg}}\xspace}

\newcommand{\ABOT}{\ensuremath{\check{\bot}}\xspace}

\newcommand{\wdia}{\ensuremath{\Diamond}\xspace}

\newcommand{\WDIA}{\ensuremath{\hat{\Diamond}}\xspace}
\newcommand{\BBOX}{\ensuremath{\check{\blacksquare}}\xspace}

\renewcommand{\epsilon}{\varepsilon}

\newcommand{\vp}{\overline{p}}

\newcommand{\ophi}{\overline{\phi}}
\newcommand{\opsi}{\overline{\psi}}

\newcommand{\bba}{\mathbb{A}}
\newcommand{\bbA}{\mathbb{A}}
\newcommand{\bbL}{\mathbb{L}}

\newcommand{\AATOP}{\hat{\top}}

\newcommand{\marginnote}[1]{\marginpar{\raggedright\tiny{#1}}}

\tikzset{
	treenode/.style = {align=center, inner sep=0pt, text centered},
	Ske/.style = {treenode, ellipse, double, draw=black,
		minimum width=6pt, thick},
	PIA/.style = {treenode, ellipse, black, draw=black,
		minimum width=6pt},
	Crit/.style = {treenode, rectangle, draw=black,
		minimum width=0.5em, minimum height=0.5em}
}

\def\fCenter{{\mbox{$\ \vdash\ $}}}

\EnableBpAbbreviations

\newcommand{\fns}{\footnotesize}

\newcommand{\commment}[1]{}
\numberwithin{equation}{section}

\newcommand{\cceq}{\coloneqq}

\newcommand{\OSI}{\overline{\Sigma}}

\newcommand{\OXI}{\overline{\Xi}}
\newcommand{\ol}[1]{\overline{#1}}
\newcommand{\vd}{\ \,\textcolor{black}{\vdash}\ \,}

\usepackage{listings}
\begin{document}

\copyrightyear{2022}
\copyrightclause{Copyright for this paper by its authors.
  Use permitted under Creative Commons License Attribution 4.0
  International (CC BY 4.0).}

\conference{CI-BD-SOQE 2026: Workshop on Craig Interpolation, Beth Definability, and Second-Order Quantifier Elimination, Lisbon, Portugal, 24–25 July 2026}

\title{Modular constructive Lyndon interpolation for nondistributive logics}

\author[1]{Andrea {De Domenico}}[%
orcid=0000-0002-8973-7011,
email=andrea.domenico@imdea.org,
]
\cormark[1]
\fnmark[1]
\address[1]{IMDEA Software Institute, Madrid, Spain}

\author[2]{Giuseppe Greco}[%
orcid=0000-0002-4845-3821,
email=g.greco@vu.nl,
]
\fnmark[1]
\address[2]{Vrije Universiteit Amsterdam, The Netherlands}

\author[2,3]{Alessandra Palmigiano}[%
orcid=0000-0001-9656-7527,
email=alessandra.palmigiano@vu.nl,
]
\fnmark[1]
\address[3]{Department of Mathematics and Applied Mathematics, University of Johannesburg, South Africa}

\author[4]{Apostolos Tzimoulis}[%
orcid=0000-0002-6228-4198,
email=apostolos@tzimoulis.eu,
]
\fnmark[1]
\address[4]{Department of Computer Science, University of Luxembourg}

\cortext[1]{Corresponding author.}
\fntext[1]{These authors contributed equally.}

\begin{abstract}
We establish the Lyndon interpolation property for basic lattice expansion logics (LE-logics) in arbitrary signatures using display calculi. Our approach is {\em constructive}, yielding interpolants algorithmically from derivations, and {\em modular}, in the sense that interpolation for axiomatic extensions can be obtained by verifying a local interpolation property for the analytic structural rules corresponding to the additional axioms. To this end, we identify a class of \emph{interpolation-safe} structural rules preserving local Lyndon interpolation. As applications of the general framework, we show that the tense version of Holliday's fundamental modal logic enjoys the Lyndon interpolation property.
\end{abstract}


\maketitle

\section{Introduction}

In its classical form, the {\em Craig interpolation property} states that whenever a formula $A \to B$ is derivable in a logic, there exists a formula $C$, called an {\em interpolant}, such that both $A \to C$ and $C \to B$ are derivable, and moreover $C$ contains only the non-logical symbols common to $A$ and $B$. More generally, in abstract formulations of interpolation, implication is replaced by the derivability relation of the logic, while the formulas $A$ and $B$ may be replaced by contexts or bunches of formulas, either in the language of the logic itself or in a suitable extension thereof. The interpolant $C$, however, remains a formula of the original language of the logic, possibly occurring within a non-empty context.  The stronger {\em Lyndon interpolation property} additionally requires preservation of the {\em polarity} of propositional variables. Interpolation has deep connections with algebra, model theory, proof theory, and category theory, as well as important applications in computer science, including modular specification, formal verification, model checking, and automated reasoning. In algebraic logic, interpolation is closely related to amalgamation and superamalgamation properties, while explicitly constructing interpolants usually requires the use of analytic calculi with good structural properties.

The study of interpolation in non-classical logics has produced a broad and technically diverse literature. Surveys such as \cite{DAgostino2008} and \cite{fussner2025interpolationnonclassicallogics} document the variety of semantic, algebraic, and proof-theoretic techniques developed for various modal and substructural logics. Among proof-theoretic approaches, a particularly fruitful line of research investigates interpolation via analytic calculi from which interpolants can be extracted algorithmically \cite{GJK2025interpolationprooftheory,Lyon2020}. 

 The results  of the present paper are closest to  Kuznets' modular approach to interpolation based on {\em labelled sequent calculi} \cite{Kuznets16}, which adapts results obtained for a class of {\em internal} sequent-like formalisms \cite{kuznets2015interpolation}. In this line of research, interpolants are constructed inductively from derivations in analytic labelled calculi, yielding constructive proofs of Craig and Lyndon interpolation. One of the key features of this method is its modularity: interpolation for axiomatic extensions can be established by verifying local conditions on the additional rules corresponding to frame conditions. In particular, Kuznets identified a broad class of rules for which interpolation follows uniformly: {\em Horn-like geometric rules}, which equivalently correspond to frame conditions that can be captured by quantifier-free Horn formulas. This programme has significantly expanded the scope of proof-theoretic interpolation methods for substructural modal logics.

The present paper develops this modular and constructive perspective in the setting of lattice expansion logics ({\em LE-logics} \cite{conradie2019algorithmic, conradie2020non}) and their proper display calculi \cite{LEframes}. A characteristic feature of  LE-logics is that the distributive laws of conjunction and disjunction are not necessarily valid, while the settings of most existing proof-theoretic results on interpolation assume the validity of distributive laws. This nondistributive setting has recently attracted increasing attention \cite{almeida2026superamalgamationmodallatticesnondistributive,Holliday2023,Holliday2024}.

Proper display calculi provide a particularly natural framework for this investigation,  because they enjoy fundamental properties such as the subformula property and cut elimination in a modular way \cite{LEframes}, and moreover a large class of axiomatic extensions of the basic LE-logics can  be endowed with analytic calculi by algorithmically generating  the analytic structural rules corresponding to the additional axioms \cite{ucaaptt, ChnGrePalTzi21}. This proof-theoretic environment makes it possible to study interpolation through purely proof-theoretic criteria on structural rules, in a way analogous to the role played by frame rules in Kuznets' approach to labelled calculi.

The main contribution of the present paper is a general constructive proof of Lyndon interpolation for basic LE-logics using display calculi, together with a modular extension method for axiomatic extensions. More precisely, we show that the display calculus $\mathbf{D.LE}$ for a basic LE-logic in any given but arbitrary signature enjoys a {\em local Maehara interpolation property}, from which effective interpolation follows. We then identify a class of structural rules, which we call \emph{interpolation-safe rules}, preserving this property. These rules play a role analogous to the Horn-like geometric rules in Kuznets' framework, but are formulated intrinsically in the setting of display calculi. As a consequence, interpolation for many axiomatic extensions of LE-logics can be established uniformly by checking local syntactic conditions on the added structural rules.

Our framework applies in particular to several recently introduced lattice-based modal logics. As illustrations, we obtain constructive Lyndon interpolation results for fundamental logic and its tense modal expansions introduced in \cite{Holliday2023,Holliday2024}. These logics arise from the study of non-distributive lattice semantics motivated both by proof theory and by applications to modal reasoning and natural language semantics.

\paragraph{Structure of the paper.} In Section \ref{sec:prelim}, we recall the necessary background on LE-logics and display calculi. In Section \ref{subsec:interpolation_proof_LE}, we provide the formal definitions of the various interpolation properties and establish the local Maehara-Lyndon interpolation property for the basic calculus $\mathbf{D.LE}$. In Section \ref{subsec: interpolation_axioms_LE}, we show how to extend the strategy to some axiomatic extensions of our base logic and showcase some examples. We conclude in Section \ref{sec:conclusions-interpolation} with some directions for further work.

\section{Preliminaries} \label{sec:prelim}

\subsection{LE-logics and their algebraic semantics}
\label{ssec: LE-logics}

An {\em order-type} over $n\in \mathbb{N}$ is an $n$-tuple $\varepsilon\in \{1, \partial\}^n$. For every order type $\varepsilon$, we denote its {\em opposite} order type by $\varepsilon^\partial$, that is, $\varepsilon^\partial_i = 1$ iff $\varepsilon_i=\partial$ for every $1 \leq i \leq n$. 
	
The language $\mathcal{L}_\mathrm{LE}(\mathcal{F}, \mathcal{G})$ (from now on abbreviated as $\mathcal{L}_\mathrm{LE}$) takes as parameters a denumerable set of proposition letters $\mathsf{AtProp}$, elements of which are denoted $p,q,r$, possibly with indexes, and disjoint sets of connectives $\mathcal{F}$ and $\mathcal{G}$.
Each $f\in \mathcal{F}$ (resp.~$g\in \mathcal{G}$) has arity $n_f\in \mathbb{N}$ (resp.\ $n_g\in \mathbb{N}$) and is associated with some order-type $\varepsilon_f$ over $n_f$ (resp.\ $\varepsilon_g$ over $n_g$). The terms (formulas) of $\mathcal{L}_\mathrm{LE}$ are defined recursively as follows: 

\begin{center}
$\varphi \cceq p \mid \bot \mid \top \mid \varphi \wedge \varphi \mid \varphi \vee \varphi \mid f(\varphi_1, \ldots, \varphi_{n_f}) \mid g(\varphi_1, \ldots, \varphi_{n_g})$
\end{center}

\noindent where $p \in \mathsf{AtProp}$, $f \in \mathcal{F}$, $g \in \mathcal{G}$. Terms in $\mathcal{L}_\mathrm{LE}$ will be denoted either by $s,t$, or by lowercase Greek letters such as $\varphi, \psi, \gamma$ etc. 
In the remainder of the paper, when it is clear from the context,
we will often simplify notation and write e.g.\  $n$ for $n_f$ and $\varepsilon_i$ for $\varepsilon_{f,i}$. We also extend the $\{1,\partial\}$-notation to the symbols $\vee,\wedge,\bot,\top,\le,\vdash$ by defining 

\begin{center}
$\vee^\partial=\wedge,\qquad \wedge^\partial=\vee,\qquad \bot^\partial=\top,\qquad \top^\partial=\bot,\qquad {\le^\partial}={\ge},\qquad \varphi\ {\vdash^\partial}\ \psi=\psi\ {\vdash}\ \varphi
$
\end{center}

\noindent while superscript $^1$ denotes the identity map. Therefore, in what follows, we will sometimes write e.g.~$\vee^{\varepsilon_i}$ to denote $\vee$ when $\varepsilon_i = 1$ and  $\wedge$ when $\varepsilon_i = \partial$. In what follows, for every $k \in \mathcal{F} \cup \mathcal{G}$ we use $k(\opsi)[\varphi]_i$ to indicate that the formula $\varphi$ occurs in the $i$-th coordinate of the vector $\opsi$.

For any  $\mathcal{L}_\mathrm{LE} = \mathcal{L}_\mathrm{LE}(\mathcal{F}, \mathcal{G})$, the {\em basic}, or {\em minimal} $\mathcal{L}_\mathrm{LE}$-{\em logic} $\mathbf{L}_\mathrm{LE}$ is a set of sequents $\varphi\vdash\psi$, with $\varphi,\psi\in\mathcal{L}_\mathrm{LE}$, which contains  the following sequents as axioms: 

\begin{center}
$\bot\vdash p, \ \ \ p\vdash p, \ \ \ p\vdash \top, \ \ \ p\vdash p \vee q, \ \ \ q\vdash p\vee q, \ \ \ p\wedge q\vdash p, \ \ \ p\wedge q \vdash q,$
\end{center}

\begin{center}
$f(\vp)[\bot^{\varepsilon_i}]_i \vdash \bot,  \qquad f(\vp)[q\vee^{\varepsilon_i} r]_i \vdash f(\vp)[q]_i \vee f(\vp)[r]_i,$
\end{center}

\begin{center}
$\top \vdash g(\vp)[\top^{\varepsilon_i}]_i, \qquad g(\vp)[q]_i \wedge g(\vp)[r]_i \vdash g(\vp)[q\wedge^{\varepsilon_i} r]_i,$
\end{center}

\noindent and is closed under the following inference rules (recall that $\varphi\vdash^\partial\psi$ means $\psi\vdash\varphi$):

\begin{center}
\begin{tabular}{@{}c@{}}
\AXC{$\varphi \vdash \chi$}
\AXC{$\chi \vdash \psi$}
\BIC{$\varphi \vdash \psi$}
\DP
 \ 
\AXC{$\varphi \vdash \psi$}
\UIC{$\varphi(\chi/p) \vdash \psi(\chi/p)$}
\DP
 \ 
\AXC{$\chi \vdash \varphi$}
\AXC{$\chi \vdash \psi$}
\BIC{$\chi \vdash \varphi \wedge \psi$}
\DP
 \ 
\AXC{$\varphi \vdash \chi$}
\AXC{$\psi \vdash \chi$}
\BIC{$\varphi \vee \psi \vdash \chi$}
\DP
 \\
 
\rule[0mm]{0mm}{8mm}\AXC{$\varphi \vdash^{\varepsilon_{f,i}} \psi$}
\UIC{$f(\vp)[\varphi]_i \vdash f(\vp)[\psi]_i$}
\DP

\qquad\qquad

\AXC{$\varphi \vdash^{\varepsilon_{g,i}} \psi$}
\UIC{$g(\vp)[\varphi]_i \vdash g(\vp)[\psi]_i$}
\DP
 \\
\end{tabular}
\end{center}


In a basic language $\mathcal{L}_{\mathrm{LE}}(\mathcal{F},\mathcal{G})$, the elements of $\mathcal{F}\cup\mathcal{G}$ are usually mutually independent. However, in some cases,  some   connectives in $\mathcal{F}\cup\mathcal{G}$ are one another's residuals (or (dual) Galois-residuals) in some coordinates.\footnote{In what follows, we will typically not make a difference between co-variant and contravariant residuation, and refer to any of these as `residuals'.} In particular, given $f\in\mathcal{F}$ or $g\in\mathcal{G}$ we  have $f^\sharp_i\in\mathcal{G}$ if $\varepsilon_{f,i} = 1$ or $g^\flat_i\in\mathcal{F}$ if $\varepsilon_{g,i} = 1$, and $f^\sharp_i\in\mathcal{F}$ if $\varepsilon_{f,i} = \partial$ or $g^\flat_i\in\mathcal{G}$ if $\varepsilon_{g,i} = \partial$, the order-type of which are (i) $\varepsilon_{f_i^\sharp,i} = \varepsilon_{f,i}$ and $\varepsilon_{f_i^\sharp,j} = (\varepsilon_{f,j})^{\varepsilon_{f,i}^\partial}$ for any $j\neq i$, and (ii) $\varepsilon_{g_i^\flat,i} = \varepsilon_{g,i}$ and $\varepsilon_{g_i^\flat,j} = (\varepsilon_{g,j})^{\varepsilon_{g,i}^\partial}$ for any $j\neq i$. For instance, if $f$ and $g$ are binary connectives such that $\varepsilon_f = (1, \partial)$ and $\varepsilon_g = (\partial, 1)$, then $\varepsilon_{f^\sharp_1} = (1, 1)$, $\varepsilon_{f^\sharp_2} = (1, \partial)$, $\varepsilon_{g^\flat_1} = (\partial, 1)$ and $\varepsilon_{g^\flat_2} = (1, 1)$.\footnote{Note that this notation depends on the connective which is taken as primitive, and needs to be carefully adapted to well known cases. For instance, consider the  `fusion' connective $\circ$ (which, when denoted  as $f$, is such that $\varepsilon_f = (1, 1)$). Its residuals
$f_1^\sharp$ and $f_2^\sharp$ are commonly denoted $/$ and
$\backslash$ respectively. However, if $\backslash$ is taken as the primitive connective $g$, then $g_2^\flat$ is $\circ = f$, and
$g_1^\flat(x_1, x_2): = x_2/x_1 = f_1^\sharp (x_2, x_1)$. This example shows
that, when identifying $g_1^\flat$ and $f_1^\sharp$, the conventional order of the coordinates is not preserved, and depends on which connective
is taken as primitive.} The intended interpretation of the $n_f$-ary connective $f^\sharp_i$ is  
the right residual of $f\in\mathcal{F}$ in its $i$th coordinate if $\varepsilon_{f,i} = 1$ (resp.\ its dual Galois-residual if $\varepsilon_{f,i} = \partial$). The intended interpretation of the  $n_g$-ary connective $g^\flat_i$ is the left residual of $g\in\mathcal{G}$ in its $i$th coordinate if $\varepsilon_{g,i} = 1$ (resp.\ its Galois-residual if $\varepsilon_{g,i} = \partial$). 
%
%
If so, $\mathbf{L}_{\mathrm{LE}}$ is augmented with the invertible rules 

\begin{center}
{\centerline{
\begin{tabular}{cc}
\AXC{$f(\ophi)[\varphi]_i \fCenter \psi$}
\doubleLine
\UIC{$\varphi \fCenter^{\varepsilon_{f,i}} f^\sharp_i(\ophi)[\psi]_i$}
\DP
\qquad & \qquad
\AXC{$\varphi \fCenter g(\ophi)[\psi]_i$}
\doubleLine
\UIC{$g^\flat_i(\ophi)[\varphi]_i \fCenter^{\varepsilon_{g,i}} \psi$}
\DP \\
\end{tabular}
}}
\end{center}


Any language $\mathcal{L}_\mathrm{LE} = \mathcal{L}_\mathrm{LE}(\mathcal{F}, \mathcal{G})$ can be associated with the language $\mathcal{L}_\mathrm{LE}^* = \mathcal{L}_\mathrm{LE}(\mathcal{F}^*, \mathcal{G}^*)$, where $\mathcal{F}^*\supseteq \mathcal{F}$ and $\mathcal{G}^*\supseteq \mathcal{G}$ are obtained by expanding $\mathcal{L}_\mathrm{LE}$ with the residuals of each connective in each coordinate. Then, the logic $\mathbf{L}_{\mathrm{LE}}$ can be expanded to the minimal fully residuated $\mathcal{L}_\mathrm{LE}^\ast$-logic $\mathbf{L}_\mathrm{LE}^*$, by adding the corresponding residuation rules.  
The logic $\mathbf{L}_\mathrm{LE}^*$ is a conservative extension of $\mathbf{L}_\mathrm{LE}$ (cf.~\cite[Theorem 2.4]{ChnGrePalTzi21}), i.e.~every $\mathcal{L}_\mathrm{LE}$-sequent $\varphi\vdash\psi$ is derivable in $\mathbf{L}_\mathrm{LE}$ iff $\varphi\vdash\psi$ is derivable in $\mathbf{L}_\mathrm{LE}^*$.

\begin{example} \label{ex:1}
    Fundamental modal logic \cite{Holliday2024} is an LE-logic in the signature $\mathcal{F} \coloneqq \{\Diamond\}$, $\mathcal{G} \coloneqq \{\neg, \Box\}$, where $n_\neg = n_\Box = n_\Diamond = 1$,  $\varepsilon_\Box = \varepsilon_\Diamond = 1$ and $\varepsilon_\neg = \partial$. It is obtained by adding the following sequents and rules to $\mathbf{L}_\mathrm{LE}$.
    \begin{center}
    $p \wedge \neg p \vdash \bot$ \quad \AXC{$\varphi \vdash \neg \psi$}\UIC{$\psi \vdash \neg \varphi$}\DP \quad $\Diamond \neg p \vdash \neg \Box p$
    \end{center}
    In this case, $\mathcal{F}^* = \{\Diamond, \blacklozenge\}$ and $\mathcal{G}^* = \{\neg, \Box, \blacksquare\}$, where $\blacklozenge$ is the left residual of $\Box$ and $\blacksquare$ is the right residual of $\Diamond$. Notice that $\neg$ is its own Galois-residual. In the following, we write \emph{tense fundamental modal logic} to denote Holliday's fundamental modal logic augmented with the residuals of the modal connectives.
\end{example}

\subsection{Proper display calculi for basic normal LE-logics} \label{ssec:display}

In this section, we recall the definition of
the proper display calculus $\mathbf{D.LE}$ for the basic normal $\mathcal{L}_{\mathrm{LE}}$-logic for a fixed but arbitrary LE-signature $\mathcal{L} = \mathcal{L}_{\mathrm{LE}}(\mathcal{F}, \mathcal{G})$  (cf.~Section \ref{ssec: LE-logics}). 
Our presentation is a more streamlined version of the one introduced in \cite{ucaaptt} for DLE-logics and then straightforwardly generalized to LE-logics in \cite{ChnGrePalTzi21}. 
 
 Let $S_{\mathcal{F}}: = \{\FH \mid f\in \mathcal{F}^*\}$ and $S_{\mathcal{G}}: = \{\GC \mid g\in \mathcal{G}^*\}$ be the sets of structural connectives indexed by  $\mathcal{F}^*$ and $ \mathcal{G}^*$ respectively. Each structural connective comes with an arity and an order-type which coincides with those of its associated operational connective in $ \mathcal{F}^*$ and $\mathcal{G}^*$.
The calculus $\mathbf{D.LE}$ manipulates sequents $\Pi \fCenter \Sigma$, where 
$\Pi$ and $\Sigma$ are structures of two sorts which are built from formulas and are defined by the following simultaneous recursion:

\begin{center}
$\mathsf{Str}_{\mathcal{F}} \ni \Pi  \cceq A \mid \AATOP \mid \FH\, (\OXI) \ \ \ \ \ \ \ \ \ \ \mathsf{Str}_{\mathcal{G}} \ni \Sigma \cceq A \mid \ABOT \mid \GC\, (\OXI)$
\end{center}

\noindent where $A$ is an $\mathcal{L}_{\mathrm{LE}}$-formula, $\FH \in S_{\mathcal{F}}$, $\GC \in S_{\mathcal{G}}$, and $\OXI\in \mathsf{Str}_{\mathcal{F}}^{\varepsilon_f}$ (resp.~$\OXI\in \mathsf{Str}_{\mathcal{G}}^{\varepsilon_g}$), 
 where $\mathsf{Str}_{\mathcal{F}}^{\varepsilon_f} = \mathsf{Str}_{\mathcal{F}}^{\varepsilon_{f,1}}\times\cdots\times \mathsf{Str}_{\mathcal{F}}^{\varepsilon_{f,n_f}}$ where $\mathsf{Str}_{\mathcal{F}}^\partial=\mathsf{Str}_{\mathcal{G}}$, and dually for $\mathsf{Str}_{\mathcal{G}}^{\varepsilon_g}$. In what follows, for every $o \in S_{\mathcal{F}} \cup S_{\mathcal{G}}$ we use $o(\OXI)[\Pi]_i$ (resp.~$o(\OXI)[\Sigma]_i$) to indicate that the structure $\Pi$ (resp.~$\Sigma$) occurs in the $i$-th coordinate of the vector $\OXI$.

If $S_1$ and $S_2$ are sequents, the  double-horizontal-line notation 
{\fns
\AXC{$S_1$}
\LL{$r$}
\doubleLine
\UIC{$S_2$}
\DP}
is an abbreviation for the rules {\fns  
\AXC{$S_1$}
\LL{$r$}
\UIC{$S_2$}
\DP} 
and 
{\fns 
\AXC{$S_2$}
\RL{$r^{-1}$}
\UIC{$S_1$}
\DP}.

\begin{itemize}
\item Identity and Cut rules:
\end{itemize}

\begin{center}
\begin{tabular}{rl}
\AXC{\phantom{$\Gamma \fCenter \varphi$}}
\LL{\fns Id$_p$}
\UI$p \fCenter p$
\DP
 & 
\AX$\Pi \fCenter A$
\AX$A \fCenter \Sigma$
\RL{\fns Cut$_A$}
\BI$\Pi \fCenter \Sigma$
\DP
 \\
\end{tabular}
\end{center}

\begin{itemize}		
\item Display rules for $f\in \mathcal{F}$ and $g\in \mathcal{G}$: for any $1\leq k \leq n_f$ and $1\leq \ell \leq n_g$,\footnote{The notation $\FH \dashv \FCS_k$ (resp.~$\GHF_{\ell} \dashv \GC$) indicates that $\FH$ and $\FCS_k$ (resp.~$\GHF_{\ell}$ and $\GC$) are in a {\em residuated pair} and $\FCS_k$ (resp.~$\GHF_{\ell}$) is the right residual (resp.~left residual) of $\FH$ (resp.~$\GC$) in the $k$-th coordinate (resp.~$\ell$-th coordinate). The notation $(\GC, \GCF_{\ell})$ (resp.~$(\FH, \FHS_k)$) indicates that $\GC$ and $\GCF_{\ell}$ (resp.~$\FH$ and $\FHS_k$) are in a {\em Galois connection} (resp.~{\em dual Galois connection}) and $\GCF_{\ell}$ (resp.~$\FHS_k$) is the right residual (resp.~left residual) of $\GC$ (resp.~$\FH$) in the $\ell$-th coordinate (resp.~$k$-th coordinate).} 
\end{itemize}

If $\varepsilon_{f,k} = 1$ and $\varepsilon_{g,\ell} = 1$

\begin{center}
\begin{tabular}{rl}
 \\
\AX$\FH\, (\OXI)[\Pi]_k \fCenter \Sigma$
\doubleLine
\LL{\fns $\FH \dashv \FCS_k$}
\UI$\Pi \fCenter \FCS_k\, (\OXI)[\Sigma]_k$
\DP
 &  
\AX$\Pi \fCenter \GC\, (\OXI)[\Sigma]_{\ell}$
\doubleLine
\RL{\fns $\GHF_{\ell} \dashv \GC$}
\UI$\GHF_{\ell}\, (\OXI)[\Pi]_{\ell} \fCenter \Sigma$
\DP 
\end{tabular}
\end{center}

If $\varepsilon_{f,k} = \partial$ and $\varepsilon_{g,\ell} = \partial$

\begin{center}
\begin{tabular}{rl}
\AX$\FH\, (\OXI)[\Sigma]_k \fCenter \Sigma'$
\doubleLine
\LL{\fns $(\FH, \FHS_k)$}
\UI$\FHS_k\, (\OXI)[\Sigma']_k \fCenter \Sigma$
\DP
 & 
\AX$\Pi' \fCenter \GC\, (\OXI)[\Pi]_{\ell}$
\doubleLine
\RL{\fns $(\GC, \GCF_{\ell})$}
\UI$ \Pi \fCenter \GCF_j\, (\OXI)[\Pi']_{\ell}$
\DP
\\
\end{tabular}
\end{center}



\begin{itemize}
\item Logical introduction rules for lattice connectives, for $i \in \{1,2\}$:
\end{itemize}

\begin{center}
\begin{tabular}{c}
\AX$\AATOP \fCenter \Sigma$
\LL{\fns $\aatop_L$}
\UI$\aatop \fCenter \Sigma$
\DP
 \
\AXC{$\phantom{\Sigma}$}
\RL{\fns $\aatop_R$}
\UI$\Pi\fCenter \aatop$
\DP
 \ \ \ 
\AX$A_i \fCenter \Sigma$
\LL{\fns $\aand_{Li}$}
\UI$A_1 \aand A_2 \fCenter \Sigma$
\DP
 \ 
\AX$\Pi \fCenter A$
\AX$\Pi \fCenter B$
\RL{\fns $\aand_R$}
\BI$\Pi \fCenter A \,\aand\, B$
\DP 
 \\
    
 \\
\AXC{$\phantom{\Pi}$}
\LL{\fns $\abot_L$}
\UI$\abot \fCenter \Sigma$
\DP
 \ 
\AX$\Pi \fCenter \ABOT$
\RL{\fns $\abot_R$}
\UI$\Pi \fCenter \abot$
\DP
 \ \ \ 
\AX$A \fCenter \Sigma$
\AX$B \fCenter \Sigma$
\LL{\fns $\aor_L$}
\BI$A \,\aor\, B \fCenter \Sigma$
\DP
 \ 
\AX$\Pi \fCenter A_i$
\RL{\fns $\aor_{Ri}$}
\UI$\Pi \fCenter A_1 \aor A_2$
\DP
 \\
\end{tabular}
\end{center}

\begin{itemize}
\item Logical introduction rules for $f\in\mathcal{F}$ and $g\in\mathcal{G}$: 
\end{itemize}
\begin{center}
\begin{tabular}{rlrl}
\bottomAlignProof
\AX$\FH\, (\overline{A}) \fCenter \Sigma$
\LL{\fns$f_L$}
\UI$f(\overline{A}) \fCenter \Sigma$
\DP
 &
\bottomAlignProof
\AxiomC{$\Big(\Xi_i \fCenter^{\!\!\varepsilon_{f,i}}\, A_i \Big)_{1\leq i\leq n_f}$}
\RL{\fns$f_R$}
\UI$\FH\, (\OXI)\fCenter f(\overline{A})$
\DP
 &
\bottomAlignProof
\AxiomC{$\Big(A_i \fCenter^{\!\!\varepsilon_{{g,i}}}\; \Xi_i \Big)_{1\leq i\leq n_g}$}
\LL{\fns$g_L$}
\UI$g(\overline{A}) \fCenter \GC\, (\OXI)$
\DP
 &
\bottomAlignProof
\AX$\Pi \fCenter \GC\, (\overline{A})$
\RL{\fns$g_R$}
\UI$\Pi \fCenter g(\overline{A})$
\DP
 \\
\end{tabular}
\end{center}
\noindent where $\OXI = (\Xi_1,\ldots, \Xi_{n_f})$ in $f_R$ and $\OXI = (\Xi_1,\ldots, \Xi_{n_g})$ in $g_L$.	
                In particular, if $f$ and $g$ are $0$-ary (i.e.~they are constants), the rules $f_R$ and $g_L$ above reduce to the axioms ($0$-ary rules) $\FH \fCenter f$ and $g \fCenter \GC$.

	The calculus $\mathbf{D.LE}$ is sound w.r.t.~the class of complete $\mathcal{L}$-algebras (cf.~\cite[Theorem 12]{ucaaptt} and \cite[Theorem 2.8]{LEframes}).
	 Moreover, $\mathbf{D.LE}$ is a proper display calculus (cf.~\cite[Theorem 26]{ucaaptt} and \cite[Appendix A]{LEframes}), and hence cut elimination holds for it as a consequence of a Belnap-style cut elimination metatheorem (cf.~\cite[Section 2.2 and Appendix A]{ucaaptt} and \cite[Theorem A.7]{LEframes}).

\paragraph{Metastructures and structural rules.} In the presentation of the language of $\mathbf{D.LE}$, $\Sigma_i$, $\Pi_i$, and $\Xi_i$ are metavariables for structures. To formally present analytic structural rules, we need to make use of metastructures, i.e., structures that are constructed by structural metavariables. Let $\mathsf{MVar} = \mathsf{MVar}_{\mathcal{F}}\uplus \mathsf{MVar}_{\mathcal{G}}$ be the denumerable set of {\em metavariables} of sorts $\Pi_1, \Pi_2, \ldots\in \mathsf{MVar}_{\mathcal{F}}$ and $\Sigma_1, \Sigma_2, \ldots \in \mathsf{MVar}_{\mathcal{G}}$.
The sets $\mathsf{MStr}_{\mathcal{F}}$ and $\mathsf{MStr}_{\mathcal{G}}$ of the $\mathcal{F}$- and $\mathcal{G}$-{\em metastructures} are defined by simultaneous induction as follows:

\begin{center}
\begin{tabular}{c}
$\mathsf{MStr}_{\mathcal{F}} \ni \Phi \ \cceq \  \Pi \mid \FH(\overline{\Phi})$\quad \quad \quad
$\mathsf{MStr}_{\mathcal{G}} \ni \Psi \ \cceq \  \Sigma \mid \GC(\overline{\Psi})$
\end{tabular}
\end{center}
where $\Pi\in \mathsf{MVar}_{\mathcal{F}}$, $\Sigma\in \mathsf{MVar}_{\mathcal{G}}$, $\FH \in \mathcal{F}^\ast$ and $\GC \in \mathcal{G}^\ast$ and $\overline{\Phi} \in \mathsf{MStr}_{\mathcal{F}}^{\varepsilon_f}$,  and $\overline{\Psi} \in \mathsf{MStr}_{\mathcal{G}}^{\varepsilon_g}$, and for any order-type $\varepsilon$ on $n$, we let $\mathsf{MStr}_{\mathcal{F}}^{\varepsilon} \coloneqq \prod_{i = 1}^{n}\mathsf{MStr}_{\mathcal{F}}^{\varepsilon_i}$ and $\mathsf{MStr}_{\mathcal{G}}^{\varepsilon} \coloneqq \prod_{i = 1}^{n}\mathsf{MStr}_{\mathcal{G}}^{\varepsilon_i}$, where for all $1 \leq i \leq n$,
\begin{center}
\begin{tabular}{ll}
$\mathsf{MStr}_{\mathcal{F}}^{\varepsilon_i} = \begin{cases} 
\mathsf{MStr}_{\mathcal{F}} &\mbox{ if } \varepsilon_i = 1\\
\mathsf{MStr}_{\mathcal{G}} &\mbox{ if } \varepsilon_i = \partial
\end{cases}\quad$
&
$\mathsf{MStr}_{\mathcal{G}}^{\varepsilon_i} = \begin{cases}
\mathsf{MStr}_{\mathcal{G}}& \mbox{ if } \varepsilon_i = 1,\\
\mathsf{MStr}_{\mathcal{F}} & \mbox{ if } \varepsilon_i = \partial.
\end{cases}$
\end{tabular}
\end{center}

\begin{definition}\label{def:AnalyticStructuralRule}
An {\em analytic structural rule} is  a rule of the form 
\begin{center}
\AXC{$(\Phi_1\fCenter \Psi_1)(\overline{\Upsilon}_1)$}
\AXC{$\cdots$}
\AXC{$(\Phi_n\fCenter \Psi_n)(\overline{\Upsilon}_n)$}
\TIC{$(\Phi_0 \fCenter \Psi_0)(\overline{\Upsilon}_0)$}
\DP
\end{center}
where $\overline{\Upsilon}_i$ with $0 \leq i \leq n$ is the set of metavariables occurring in each sequent $\Phi_i \fCenter \Psi_i$, $\overline{\Upsilon}_0 \supseteq \overline{\Upsilon}_1 \cup \ldots \cup \overline{\Upsilon}_n$, and while structural metavariables might occur multiple times in the premises they occur only once in the conclusion. 
An instance of the rule $R$ is obtained from $R$ by uniformly substituting each structural metavariable $\Pi\in \mathsf{MVar}_{\mathcal{F}}$ with an element of $\mathsf{Str}_{\mathcal{F}}$ and every structural metavariable $\Sigma\in \mathsf{MVar}_{\mathcal{G}}$ with an element of $\mathsf{Str}_{\mathcal{G}}$. A calculus contains the analytic rule $R$ if it contains every instance of $R$.
\end{definition}


The expressivity of proper display calculi has been characterized in terms of the  {\em analytic inductive} axioms (cf.~\cite[Definition 55]{ucaaptt}),  a proper subclass of the class of axioms to which the state of the art algorithmic correspondence theory via second-order quantifier elimination applies:

\begin{prop}(cf.~\cite[Proposition 59, Proposition 61]{ucaaptt})
	\label{prop:type5} Every analytic inductive LE-inequality can be equivalently transformed, via an ALBA-reduction, into a set of analytic structural rules. Conversely, every analytical structural rule is equivalent to an inductive LE-inequality.
\end{prop}

\subsection{Auxiliary notations and lemmas}
In this section, we gather some notation and facts that will be useful throughout the paper.

\begin{definition} \label{def:display_equivalence}
    The $\mathbf{D.LE}$-sequents $\Pi_1 \vdash \Sigma_1$ and  $\Pi_2 \vdash \Sigma_2$ are \emph{display equivalent} (notation: $\Pi_1 \vdash \Sigma_1 \equiv_D \Pi_2 \vdash \Sigma_2$) if $\Pi_1 \vdash \Sigma_1$ can be derived from $\Pi_2 \vdash \Sigma_2$ using only display postulates.
\end{definition}

It immediately follows from the definition that $\equiv_D$ is an equivalence relation. 

When a structure is denoted with a capital letter, e.g.\ $X,U,\Pi,\ldots$, we will denote with the lower-case counterpart $x,u,\pi,\ldots$ the formula that arises by substituting the structural connectives of the structure with their operational counterparts. 

We will call a structure $X$ a substructure of the sequent $\Pi\vdash\Sigma$ if $X$ is a substructure of $\Pi$ or a substructure of $\Sigma$, and denote this relationship by $(\Pi\vdash\Sigma)[X]$. In this notation, $X$ is understood to always refer to a concrete instance of the substructure $X$ in the sequent $\Pi\vdash\Sigma$. We will write $(\Pi\vdash\Sigma)[Y/X]$ to denote the substitution of the instance of the substructure $X$ in $\Pi\vdash\Sigma$ with the structure $Y$. 

The {\em display property} states that for every substructure $X$ of $\Pi\vdash\Sigma$, there exists a display equivalent sequent either of the form $X\vdash\Sigma'$ or 
of the form $\Pi'\vdash X$, such that for any $\mathcal{F}$-structure $\Gamma$ (resp.\ $\mathcal{G}$-structure $\Delta$) the sequent $\Gamma\vdash\Sigma'$ (resp.\ $\Pi'\vdash\Delta$) is display equivalent to $(\Pi\vdash\Sigma)[\Gamma/X]$ (resp.\ $(\Pi\vdash\Sigma)[\Delta/X]$). 

Given $(\Pi\vdash\Sigma)[X]$, we define $\varepsilon^{\Pi\vdash\Sigma}_X=1$ if $\Pi\vdash\Sigma$ is display equivalent to $X\vdash\Sigma'$ and $\varepsilon^{\Pi\vdash\Sigma}_X=\partial$ if $\Pi\vdash\Sigma$ is display equivalent to $\Pi'\vdash X$. In this notation, we will often omit the superscript when it is clear from the context.

The following lemma shows the inverse direction of the display property:
\begin{lemma}\label{lem:four_possible_decompositions}
   If $\Pi \vdash \Sigma $ and $\Pi' \vdash \Sigma'$ are display-equivalent, then one of the following cases holds.
   \begin{enumerate}
       \item  The structure $\Pi'$ is a substructure of $\Pi\vdash\Sigma$ and for every $\mathcal{F}$-structure $\Gamma$, $\Gamma\vdash \Sigma'$ is display equivalent to $\Pi\vdash\Sigma [\Gamma/\Pi']$.
       \item The structure $\Sigma'$ is a substructure of $\Pi\vdash\Sigma$ and for every $\mathcal{G}$-structure $\Delta$, $\Pi'\vdash \Delta$ is display equivalent to $\Pi\vdash\Sigma [\Delta/\Sigma']$. 
       \end{enumerate}
   \end{lemma}
\begin{proof}
By assumption, a derivation exists of $\Pi'\vdash \Sigma'$ from $\Pi\vdash \Sigma$ that consists only of applications of display postulates. We proceed by induction on the length $n$ of this derivation with the following stronger induction hypothesis: In $\Pi'\vdash\Sigma'$ the instantiations of all the metavariables with the exception of exactly one that appear in the last display rule that was applied, are substructures of $\Pi\vdash\Sigma$. If $X_0$ is the metavariable whose instantiation in $\Pi'\vdash\Sigma'$ is not a substructure of $\Pi\vdash\Sigma$, then, by applying to $\Pi'\vdash\Sigma'$ the display rule that puts $X_0$ on display, we obtain a sequent $X_0\vdash^{\varepsilon_{X_0}}\Delta$, 
that has already appeared in the proof. 

We can assume without loss of generality that every sequent appears only once in the proof; otherwise, the proof can be shortened. In this context in particular, the induction hypothesis implies that the instantiation of the variable in display in the last application of the rule is a substructure of $\Pi\vdash\Sigma$. Furthermore, the induction hypothesis implies that the structure that instantiates the metavariable on display in the last application of a rule in the proof always appears as a substructure that instantiates a variable. 
Hence, we can substitute it in the inverse proof with any other substructure. 

Now we can proceed with the induction proof. The statement holds vacuously for a proof of size $0$.
Let us assume that  $n\geq 1$, and that the statement holds for all derivations of length $k<n$.  If the last rule applied is $\FH\dashv \check{f}^\sharp_i$, then the derivation looks as follows:
\begin{center}
\AXC{$\Pi \fCenter \Sigma$}

\UIC{$\vdots$}
\UI$\FH(\overline{\Xi})[\Pi']_i \fCenter T$
\LL{\fns$\FH\dashv \check{f}^\sharp_i$}
\UI$\Pi' \fCenter \check{f}^\sharp_i(\overline{\Xi})[T]_i$
\DP     
\end{center} 
If $\FH(\overline{\Xi})[\Pi']_i$ corresponded to an instantiation of a metavariable in the previous application of a display rule, then by the induction hypothesis, it is a substructure of $\Pi\vdash\Sigma$, in which case $\overline{\Xi}$ are all substructures of $\Pi\vdash\Sigma$, and so is $\Pi'$ and applying a display rule that puts $T$ on display yields a sequent that has already appeared in the proof, namely $\FH(\overline{\Xi})[\Pi']_i \fCenter T$.

If $\FH(\overline{\Xi})[\Pi']_i$ doesn't correspond to an instantiation of a metavariable in the previous application of a display rule, then $\Pi'$, $\overline{\Xi}$, and $T$ do. 
In this case, the instantiation that doesn't correspond to a substructure of $\Pi\vdash\Sigma$ cannot be $\Pi'$, since then the sequent $\Pi' \fCenter \check{f}^\sharp_i(\overline{\Xi})[T]_i$ has already appeared in the proof, a possibility we have excluded. Hence, the required property also holds in this case.
In case of different display rules we argue similarly. This concludes the induction step and the proof.
\end{proof}

\section{Interpolation properties in display calculi}\label{subsec:interpolation_proof_LE}

In this section, we introduce the refinements of the Lyndon and Maehara interpolation notions in the context of display calculi, and prove that in this context the two notions coincide. We show that the rules of the basic normal LE-logic $\mathbf{L}_{\mathrm{LE}}$ in any LE-language preserve the existence of such interpolants, made precise through the notion of the local interpolation property. Through these results, we show that the Maehara  interpolation properties hold for $\mathbf{L}_{\mathrm{LE}}$ in any LE-language.

\subsection{Lyndon and Maehara interpolation}

  \begin{definition}\label{def:local_interpolant_cl}
  A {\em Lyndon interpolant} for a $\mathbf{D.LE}$ sequent $\Pi \vdash \Sigma$ is an $\mathrm{LE}$-formula $\gamma$ such that both $\Pi \vdash \gamma$ and $\gamma \vdash \Sigma$ are derivable in $\mathbf{D.LE}$; moreover, $\mathrm{Var}^+(\gamma) \subseteq \mathrm{Var}^+(\Pi) \cap \mathrm{Var}^+(\Sigma)$, and $\mathrm{Var}^-(\gamma) \subseteq \mathrm{Var}^-(\Pi) \cap \mathrm{Var}^-(\Sigma)$, where $\mathrm{Var}^+(\Upsilon)$ (resp.~$\mathrm{Var}^-(\Upsilon)$) is the set of the atomic variables occurring positively (resp.~negatively) in $\Upsilon$. The Lyndon interpolation property holds for a sequent $\Pi \vdash \Sigma$ if  a Lyndon interpolant for it exists.
\end{definition}

 \begin{definition} \label{def:local_interpolant_ml}
The {\em Maehara interpolation property} holds for a $\mathbf{D.LE}$ sequent $\Pi \vdash \Sigma$  if for every substructure $X$ occurring in the sequent, some $\mathcal{L}$-formula $\gamma_X$ exists such that both $X \vdash^{\varepsilon_X} \gamma_X$ and $(\Pi \vdash \Sigma)[\gamma_X/X]$ are derivable in $\mathbf{D.LE}$, and $\mathrm{Var}^+(\gamma_X) \subseteq \mathrm{Var}^+(X) \cap \mathrm{Var}^+((\Pi \vdash \Sigma)[[-]/X])$, and $\mathrm{Var}^-(\gamma_X) \subseteq \mathrm{Var}^-(X) \cap \mathrm{Var}^-((\Pi \vdash \Sigma)[[-]/X])$.
  \end{definition}

\begin{definition}\label{def:local_interpolation_property}
The {\em local} Lyndon/Maehara interpolation property holds for a calculus  if it is preserved by each  rule of the calculus; i.e.,  if it holds for (any instantiation of) the premises of any given rule, then it does for (the instantiation of) the conclusion.
\end{definition}

\begin{lemma}\label{lem:displaycraigtomaehara}
    A $\mathbf{D.LE}$ sequent $\Pi \vdash \Sigma$ has  the Maehara interpolation property if and only if every sequent, $\Pi' \vdash \Sigma'$, display equivalent to $\Pi \vdash \Sigma$, has a Lyndon interpolant.
\end{lemma}
\begin{proof}
    Assume that $\Pi \vdash \Sigma$ has the local Maehara interpolation property and let $\Pi'\vdash \Sigma'$ be display equivalent to $\Pi \vdash \Sigma$ . By Lemma \ref{lem:four_possible_decompositions}, it follows that either $\Pi'$ or $\Sigma'$ is a substructure of $\Pi\vdash\Sigma$. Let's assume that we are in case 1 of Lemma \ref{lem:four_possible_decompositions}, and so, in particular, $\Pi'$ is a substructure of $\Pi\vdash\Sigma$. By assumption, there exists an interpolant $\gamma$ such that $\Pi'\vdash\gamma$ and $(\Pi\vdash\Sigma) [\gamma/\Pi']$ are derivable. Then by Lemma \ref{lem:four_possible_decompositions} we have that $\gamma\vdash \Sigma'$ is display equivalent to $(\Pi\vdash\Sigma) [\gamma/\Pi']$. Hence, $\gamma$ is a Lyndon interpolant of   $\Pi'\vdash \Sigma'$. For the other case we work analogously.

    Now assume that every sequent, $\Pi'\vdash \Sigma'$, display equivalent to $\Pi\vdash\Sigma$ has a Lyndon interpolant, and let $X$ be a substructure of $\Pi\vdash\Sigma$. Let's assume without loss of generality that $\varepsilon_\Delta=1$, By the display property, there exists a sequent, $X\vdash\Sigma'$, which is display equivalent to $\Pi\vdash\Sigma$, such that for every $\mathcal{F}$-substructure $\Gamma$, $\Gamma\vdash\Sigma'$ is display equivalent to $(\Pi\vdash\Sigma)[\Gamma/X]$. By assumption,  $X\vdash\Sigma'$ has a Lyndon interpolant $\gamma$, hence $X\vdash\gamma$ and $\gamma\vdash\Sigma'$ are derivable. The second sequent is display equivalent to $(\Pi\vdash\Sigma)[\gamma/X]$. Hence, $\Pi\vdash\Sigma$ has the Maehara interpolation property.
\end{proof}

\subsection{The interpolation result}

\begin{lemma}
   The local Lyndon interpolation property holds for a display calculus if and only if the local Maehara interpolation property holds for it.
\end{lemma}
\begin{proof}
    If the local Lyndon interpolation property holds for a display calculus $D$, then the display rules of $D$ preserve the local Lyndon interpolation property. Hence, if the Lyndon interpolation property holds for a sequent $\Pi\vdash\Sigma$, it also holds for all its display-equivalent sequents, and hence, by Lemma \ref{lem:displaycraigtomaehara}, the Maehara interpolation property also holds for $\Pi\vdash\Sigma$. Hence, the Maehara interpolation property is preserved by the rules of the calculus. 
\end{proof}

\begin{lemma}\label{lem:substructures of variables}
    Let \begin{center}
\AXC{$\Phi_1\fCenter \Psi_1$}
\AXC{$\cdots$}
\AXC{$\Phi_n\fCenter \Psi_n$}
\LL{\fns R}
\TIC{$\Phi_0 \fCenter \Psi_0$}
\DP
\end{center}
be a rule of a display calculus such that each structural metavariable occurs at most once  in each premise\footnote{Notice that here we do not refer to analytic or structural rules only, hence the basic logical introduction rules are also included.}. Consider an instantiation of $R$, \begin{center}
\AXC{$\Pi_1\fCenter \Sigma_1$}
\AXC{$\cdots$}
\AXC{$\Pi_n\fCenter \Sigma_n$}
\LL{\fns $R_{inst}$}
\TIC{$\Pi_0 \fCenter \Sigma_0$}
\DP
\end{center} and let $\Gamma$ be an $\mathcal{F}$-substructure of the instantiation of a metavariable $X$ of $R$ in $R_{inst}$. If, for every $\Phi_i\vdash \Psi_i$ in which $X$ occurs, 
some $\gamma_i$ exists such that $\Gamma\vdash\gamma_i$ and $(\Pi_i\vdash \Sigma_i)[\gamma_i/\Gamma]_k$ are derivable, then  some $\gamma_0$ exists such that $\Gamma\vdash\gamma_0$ and $(\Pi_0\vdash \Sigma_0)[\gamma_0/\Gamma]_k$ are derivable. The dual statement holds for a $\mathcal{G}$-substructure $\Delta$.
\end{lemma}
\begin{proof}
    Let 
    $S=\{i \mid \text{ the metavariable $X$ occurs in }\Phi_i\vdash \Psi_i\}$. Let $\gamma_0: = \bigwedge_{i\in S}\gamma_i$. Using the right introduction for $\land$ we conclude that $\Gamma\vdash\bigwedge_{i\in S}\gamma_i$. Using display rules and the left introduction of $\land$ we obtain  $(\Pi_i\vdash \Sigma_i)[\bigwedge_{i\in S}\gamma_i/\Gamma]_k$ for every $i\in S$. Applying $R$ to these assumptions we conclude  $(\Pi_0\vdash \Sigma_0)[\gamma_0/\Gamma]_k$. If $\Delta$ is a  $\mathcal{G}$-substructure of the instantiation of a metavariable $X$, we use the introduction rules for $\lor$ to obtain $\delta_0=\bigvee_{i\in S}\delta_i$.
\end{proof}
\begin{thm} \label{thm:interpolation_LE}
The local Maehara-Lyndon interpolation property holds for the calculus $\mathbf{D.LE}$ associated with the basic LE-logic of any LE-language $\mathcal{L}$.  
\end{thm}
\begin{proof}
We need to show that the local Maehara-Lyndon interpolation property holds for each rule of $\mathbf{D.LE}$. Notice that all the basic rules satisfy the condition of Lemma \ref{lem:substructures of variables}. Hence we only need concern ourselves with structures in the conclusion that are not substructures of structural metavariables. First, let us consider the zero-ary rules:
    \begin{itemize}
    \item 
    \begin{tabular}{ccc}
        \AXC{\quad}
		\LL{\fns Id}
		\UIC{$p \vdash p$}
		\DP
        &
        \AXC{\quad}
		\RL{\fns $\top_R$}
		\UIC{$p \vdash \top$}
		\DP
        &
        \AXC{\quad}
		\LL{\fns $\bot_L$}
		\UIC{$\bot \vdash p$}
		\DP
        \end{tabular}
    \end{itemize}
There are no proper substructures in these sequents, hence Craig and Maehara interpolation coincide. The local interpolants are $p$, $\top$, and $\bot$, respectively, which clearly also satisfy the Lyndon conditions for the variables.

Now we need to consider all the other rules and show that, assuming there is an interpolant for each premise, how to construct the interpolant for the conclusion. Let us consider the display postulates for a structural $\mathcal{F}$-connective which is monotone in its $k$ coordinate, i.e.~$\varepsilon_{f,k} = 1$:
\begin{itemize}
\item  \begin{tabular}{c}
\AX$\FH\, (\OXI)[\Pi]_k \fCenter \Sigma$
\doubleLine
\LL{\fns $\FH \dashv \FCS_k$}
\UI$\Pi \fCenter \FCS_k\, (\OXI)[\Sigma]_k$
\DP 
\end{tabular}
\end{itemize}
The only case not covered by Lemma \ref{lem:substructures of variables} is when $\Delta = \FCS_k\, (\OXI)[\Sigma]_k$. In this case, by assumption we have an interpolant $\gamma$ such that $\Pi\vdash\gamma$ and $\FH\, (\OXI)[\gamma]_k \fCenter \Sigma$ are derivable. Then, by applying the display rule $\FH \dashv \FCS_k$ from premise to conclusion, we have that $\gamma\vdash \FCS_k\, (\OXI)[\Sigma]_k$. Given we already know that $\Pi \vdash \gamma$, it follows that $\gamma$ is also an interpolant of the sequent in the conclusion. Hence, this display rule preserves the local Maehara-Lyndon interpolation property. The other display rules are shown analogously.

We now consider the binary rules introducing conjunction and disjunction:
     \begin{itemize}
        \item 
        \begin{tabular}{cc}
        \AX$\varphi \fCenter \Sigma$
		\AX$\psi \fCenter \Sigma$
		\LL{\fns $\aor_L$}
		\BI$\varphi \aor \psi \fCenter \Sigma$
		\DP
        & 
        \AX$\Pi\fCenter\varphi $
		\AX$\Pi\fCenter\psi $
		\RL{\fns $\aand_R$}
		\BI$\Pi\fCenter\varphi \aand \psi $
		\DP
        \end{tabular}
    \end{itemize}
Let us focus on the right introduction $\land_R$. The only case not covered by Lemma \ref{lem:substructures of variables} is when $\Delta = \varphi\land\psi$. Clearly, $\delta_\varphi\land\delta_\psi$ is the interpolant, where $\delta_\varphi$ is the interpolant for $\varphi$ and $\delta_\psi$ is the interpolant for $\psi$ in the premises. The proof for the $\lor_L$ is analogous.
    
      \begin{itemize}
        \item 
        \begin{tabular}{cccc}
          \AX$\varphi \fCenter \Sigma$
		\LL{\fns $\aand_{L_1}$}
		\UI$\varphi \aand \psi \fCenter \Sigma$
		\DP   
        &  
        \AX$\psi \fCenter \Sigma$
		\LL{\fns $\aand_{L_2}$}
		\UI$\varphi \aand \psi \fCenter \Sigma$
		\DP
        & 
        \AX$\Pi \fCenter \varphi $
		\RL{\fns $\aor_{R_1}$}
		\UI$\Pi \fCenter\varphi \aor \psi $
		\DP
        & 
        \AX$\Pi \fCenter \psi $
		\RL{\fns $\aor_{R_2}$}
		\UI$\Pi \fCenter\varphi \aor \psi $
		\DP           
        \end{tabular}
        
    \end{itemize}  
    
  For the left introduction of $\land$, the only non-trivial substructure is $\varphi\land \psi$. But then the same interpolant for $\psi$ in the assumption works. The other rules are shown similarly.

\begin{itemize}
        \item 
        \begin{tabular}{cccc}
          \AX$\AATOP \fCenter \Sigma$
		\LL{\fns $\top_{L}$}
		\UI$\top \fCenter \Sigma$
		\DP   
        &  
        \AX$\Pi \fCenter \ABOT$
		\RL{\fns $\bot_{R}$}
		\UI$\Pi \fCenter \bot$
		\DP
        & 
        \AX$\FH(\overline{\varphi})\fCenter \Sigma$
		\LL{\fns $f_{L}$}
		\UI$f(\overline{\varphi})\fCenter \Sigma$
		\DP
        & 
        \AX$\Pi \fCenter \GC(\overline{\varphi})$
		\RL{\fns $g_{R}$}
		\UI$\Pi \fCenter g(\overline{\varphi}) $
		\DP           
        \end{tabular}
        
    \end{itemize}    
These cases are similar to the previous case, and are hence omitted.
    \begin{itemize}
    \item 
    \begin{tabular}{cc}
       \AX$\AATOP \fCenter \Sigma$
		\LL{\fns $\aatop_W$}
		\UI$\Pi \fCenter \Sigma$
		\DP  
        &   \AX$\Pi  \fCenter \ABOT$
		\RL{\fns $\abot_W$}
		\UI$\Pi \fCenter \Sigma$
		\DP        
    \end{tabular}      
    \end{itemize}
The only case not covered by Lemma \ref{lem:substructures of variables} in $\aatop_W$ is when $\Delta$ is a substructure of $\Pi$. Then the interpolant is $\top^{\varepsilon_\Delta}$. Indeed, $\Delta\vdash^{\varepsilon_\Delta}\top^{\varepsilon_\Delta}$ is derivable, and so is $(\Pi\vdash\Sigma)[\top^{\varepsilon_\Delta}/\Delta]$, by an application of  $\aatop_W$. The case for $\abot_W$ is similar.

   \begin{itemize}
        \item \AxiomC{$\Big(\Upsilon_i \fCenter \varphi_i \quad \varphi_j \fCenter \Upsilon_j \mid 1\leq i, j\leq n_f, \varepsilon_{f}(i) = 1\mbox{ and } \varepsilon_{f}(j) = \partial\Big)$}
		\RL{\fns$f_R$}
		\UI$\FH\, (\Upsilon_1,\ldots, \Upsilon_{n_f})\fCenter f(\varphi_1,\ldots, \varphi_{n_f})$
		\DP
    \end{itemize}
The only cases not covered by Lemma \ref{lem:substructures of variables} are for the substructures $\FH\,(\Upsilon_1,\ldots, \Upsilon_{n_f})$ and $f(\varphi_1,\ldots, \varphi_{n_f})$. They will have the same interpolant, in particular, if $\sigma_i$ is the interpolant of $Y_i$, then the interpolant is $f(\sigma_1,\ldots, \sigma_{n_f})$. Verifying the conditions is routine. This concludes the proof.
\end{proof}

\begin{corollary}[Local interpolation]
    For any LE-language $\mathcal{L}_{\mathrm{LE}}$, any $\mathbf{D.LE}$-derivable sequent has the local Maehara-Lyndon interpolation property, and the interpolants can be effectively computed.
\end{corollary}
\begin{proof}
Let $X\vdash Y$ be a derivable sequent and proceed by induction on the height of the derivation. The base cases are treated as in the proof of Theorem \ref{thm:interpolation_LE}. For the inductive step, if $R$  is the last rule applied in the derivation of $X\vdash Y$, then the existence of  a (local) interpolant for $X\vdash Y$ immediately follows from the inductive hypothesis and the local interpolation property holding for $R$ (cf.~Theorem \ref{thm:interpolation_LE}). 
\end{proof}

Applying the previous corollary to  $\mathbf{D.LE}$-sequents without structural connectives immediately yields the following 

\begin{corollary}[Lyndon interpolation]
    For any LE-language $\mathcal{L}_{\mathrm{LE}}$ and any $\mathcal{L}_{\mathrm{LE}}$-sequent $\varphi\vdash \psi$, if $\mathbf{L}_{\mathrm{LE}}\models \varphi\vdash \psi$, then $\mathbf{L}_{\mathrm{LE}}\models \varphi\vdash \gamma$ and $\mathbf{L}_{\mathrm{LE}}\models \gamma\vdash \psi$ for some $\gamma\in \mathcal{L}_{\mathrm{LE}}$ such that $\mathrm{Var}^+(\gamma) \subseteq \mathrm{Var}^+(\varphi) \cap \mathrm{Var}^+(\psi)$, and $\mathrm{Var}^-(\gamma) \subseteq \mathrm{Var}^-(\varphi) \cap \mathrm{Var}^-(\psi)$, and $\gamma$ 
    can be effectively computed.
\end{corollary}

\section{Generalizing interpolation to axiomatic extensions} \label{subsec: interpolation_axioms_LE}

The proof strategy of the previous section can be extended to certain axiomatic extensions of $\mathbf{L}_{\mathrm{LE}}$ equivalently presented by extending $\mathbf{D.LE}$ with appropriate rules that preserve both cut-eliminability and interpolation.

\begin{definition}\label{def:InterpolationSafeRules}
For any LE-language, a {\em special} rule in the language of $\mathbf{D.LE}$ is an analytic structural rule (see Definition \ref{def:AnalyticStructuralRule}) of one of the following general forms:
\begin{center}
    \begin{tabular}{cc}
    \AXC{$\Pi \vdash \Psi_1$} 
    \AXC{$\cdots$}
    \AXC{$\Pi \vdash \Psi_n$}
    \TIC{$\Pi \vdash \Psi_0$}
    \DP
    &
    \AXC{$\Phi_1\vdash \Sigma$} 
    \AXC{$\cdots$}
    \AXC{$\Phi_n\vdash \Sigma$}
    \TIC{$\Phi_0\vdash \Sigma$}
    \DP
    \\         
    \end{tabular}
\end{center}
where  $\Pi$ (resp.~$\Sigma$) is an atomic metavariable for structures which does not occur in any $\Psi_i$ (resp.~$\Phi_i$) for any $0 \leq i \leq n$.

An {\em interpolation-safe} rule is a special rule such that each $\Psi_i$ and $\Phi_i$  for any $1 \leq i \leq n$ contain at most one occurrence of at most one metavariable for structures and possibly zeroary structural connectives.
\end{definition}

For instance, all analytic-inductive modal reduction principles can be equivalently captured by interpolation-safe rules. An example is the well-known axiom $\Diamond \Box p \vdash \Box \Diamond p$, which is equivalently captured by the following rule:

\begin{center}
    \AXC{$X \vdash \check{\Box}\check{\blacksquare} Y$}
    \UIC{$X \vdash \check{\blacksquare}\check{\Box} Y$}
    \DP
\end{center}
Another example of this type is the axiom $\Diamond \neg p \vdash \neg \Box p$ of modal fundamental logic (cf.~Example \ref{ex:1}), which is equivalently captured by the following interpolation-safe rule:
\begin{center}
    \AXC{$X \vdash \check{\neg}\hat{\blacklozenge} Y$}
    \UIC{$X \vdash \check{\blacksquare}\check{\neg} Y$}
    \DP
\end{center}

\begin{proposition}
    \label{lem:localmaeeasyr}
Every interpolation-safe rule preserves the local Maehara-Lyndon interpolation property, and the interpolant of the conclusion is built from the interpolants in the premises using $\aand$, $\aor$, or the logical counterpart of structural connectives occurring in the conclusion.
\end{proposition}

\begin{proof}
We show that if (the instantiations of) the premises of the rule have the Maehara interpolation property, then so does (the instantiation of) the conclusion. Moreover, the interpolant is built from the interpolants of the premises only using the logical counterparts of structural connectives occurring in the conclusion of the rule.
    
Let us assume a basic LE-logic in an arbitrary but fixed language $\mathcal{L}$. Below, let us consider an instance of the first interpolation-safe structural rule described in Definition \ref{def:InterpolationSafeRules} (the other case is analogous) in the language $\mathcal{L}$: 
\begin{center}
\begin{tabular}{c}
\AXC{$\Gamma \vdash \Delta_1$} 
\AXC{$\cdots$}
\AXC{$\Gamma \vdash \Delta_n$}
\TIC{$\Gamma \vdash \Delta_0$}
\DP
\\
\end{tabular}
\end{center} 
Let $Y$ be a substructure of the sequent $\Gamma \vdash \Delta_0$.  If $Y$ is a substructure of a metavariable $X_0$ in $W\vdash T_0$ then Lemma \ref{lem:substructures of variables} gives the desired result. The other possibility is that there exists a substructure $U[X_1,\ldots,X_k]$ of $\Pi\vdash \Psi_0$ whose instantiation is $Y$, where the $X_j$, $1\leq j\leq k$, are an exhaustive list of the metavariables $U$ contains. Let's denote  by $Y^j$ the substructure of $Y$  that instantiates $X_j$.  Notice, in particular, that $\Pi$ cannot be one of the $X_j$. This implies that, since the rule is interpolation safe, each premise $\Pi\vdash \Psi_i$ has at most one instance of one variable among $X_j$ that instantiates a substructure of $Y$. Let's denote this variable by $Z_i$, and the substructure of $Y$ it instantiates by $Y_i$. 

By assumption, for the instantiation of each premise, there is an interpolant $\sigma_i$, such that $Y^j\vdash^{\varepsilon_{X_j}}\sigma_i$ and $\Gamma\vdash\Delta_i\ [\sigma_i/Y_i]_k $, where $\varepsilon_{X_j}=1$ if $X_j$ is an $\mathcal{F}$-metavariable and $\varepsilon_{X_j}=\partial$ otherwise. Consider $I_j:=\{i\ \mid\ Z_i=X_j\}$, for $1\leq j\leq k$. Then it is routine to verify, using the introduction rules for $\land$ and $\lor$, that $Y_j\vdash^{\varepsilon_{X_j}}\bigwedge^{{\varepsilon_{X_j}}}_{i\in I_j}\sigma_i$ and $\Gamma\vdash\Delta_i\ [\bigwedge^{{\varepsilon_{X_j}}}_{i\in I_j}\sigma_i/Y_i]_k$. Then applying the rule we obtain  $\Gamma\vdash\Delta_0 [U[\bigwedge^{{\varepsilon_{X_1}}}_{i\in I_1}\sigma_i,\ldots,\bigwedge^{{\varepsilon_{X_k}}}_{i\in I_k}\sigma_i]/Y]$.  The formula $u[\bigwedge^{{\varepsilon_{X_1}}}_{i\in I_1}\sigma_i,\ldots,\bigwedge^{{\varepsilon_{X_k}}}_{i\in I_k}\sigma_i]$ is the interpolant. Since each $\sigma_i$ satisfies the conditions on the restrictions on variables so do these formulas. Notice in particular that if some $I_j=\varnothing$, then the construction above implies that $\bigwedge^{{\varepsilon_{X_j}}}_{i\in I_0}\sigma_i$ is $\top^{\varepsilon_{X_j}}$. The fact that added connectivess are only those that appear as structural connectives in the conclusion of the rule is immediate. This concludes the proof.
\end{proof}

\begin{theorem}
    Every analytic axiomatic extension of $\mathbf{L}_{\mathrm{LE}}$ whose corresponding rule is a special interpolation-safe rule has the Maehara interpolation property. 
\end{theorem}
\begin{proof}
    In view of Theorem \ref{thm:interpolation_LE} and Lemma \ref{lem:localmaeeasyr} what is left to show is that the connectives of the interpolants are in the language of $\mathbf{L}_{\mathrm{LE}}$.
     Notice preliminarily that the conclusion of any analytic rule that corresponds to an axiom of  $\mathbf{L}_{\mathrm{LE}}$ contains structural connectives whose formula counterparts are part of $\mathbf{L}_{\mathrm{LE}}$.

    We proceed by induction on the size of the proof. By Lemma \ref{lem:localmaeeasyr} and the above observation, an application of a non-basic analytic rule is only adding to the interpolant connectives in $\mathbf{L}_{\mathrm{LE}}$. Inspecting the cases in Theorem \ref{thm:interpolation_LE}, reveals that connectives outside $\land$ and $\lor$ are added only when an introduction rule is applied, and the connective that is added is the one the rule introduces. This concludes the proof.
\end{proof}

We finish this section by showing that Lyndon interpolation property holds for fundamental logic and its basic modal expansions with unary {\em tense} connectives  \cite{Holliday2023,Holliday2024} (cf.~Example \ref{ex:1}). 
Indeed, if e.g.~the Galois-residuals of $\Box$ and $\Diamond$ in Example \ref{ex:1} are present in the language, then  $\mathcal{L}^\ast = \mathcal{L}$.
The rules corresponding to all the axioms of these logics are interpolation-safe except for  
 the one corresponding to the $\mathrm{LE}$-axiom $p\wedge \neg p \vdash \bot$:

\begin{center}
    \AXC{$ X \vdash \NEG X$}
    \LL{$R$}
    \UIC{$ X \vdash Y$}
    \DP
\end{center}

\begin{proposition} If   
     all  connectives in $\mathcal{L}$ are unary except conjunction and disjunction, then the local Maehara interpolation property holds for $R$. The interpolant is a formula in $\mathcal{L}^\star$. Furthermore, if $\Gamma\vdash\Delta [\Sigma]$ is an instance of the conclusion, and the display equivalent sequent $\Sigma\vdash^{\varepsilon_{\Sigma}}\Delta'$ contains only structural connectives from the language of $\mathcal{L}$, then the interpolant of $\Sigma$ is in the language $\mathcal{L}$.
\end{proposition}
\begin{proof}
    Let $\Gamma\vdash\Delta$ be an instance of the conclusion of $R$, in which case the assumption is $\Gamma\vdash \NEG\Gamma$. If $\Sigma$ is a substructure of $\Delta$, then clearly $\top^{\varepsilon(\Sigma)}$ is the interpolant, since $\Sigma\vdash^{\varepsilon(\Sigma)}\top^{\varepsilon(\Sigma)}$ is derivable. If  $\Gamma=\Pi[\Sigma]$, we can apply $R$ again and obtain $\Pi[\Sigma]\vdash\bot$. Then  $\Sigma\vdash^{\varepsilon_{\Sigma}} \Pi'[\bot]$ and so $\Sigma\vdash^{\varepsilon_{\Sigma}} \pi'[\bot]$. Since $\mathbf{L}$ contains only unary connectives $\pi'[\bot]$ contains no variables. Clearly $\pi'[\bot]\vdash^{\varepsilon_{\Sigma}}\Pi'[X]$, and so $\Pi[\pi'[\bot]]\vdash X$, for every structure $X$. In particular, $\Pi[\pi'[\bot]]\vdash\Delta$. Hence  $\pi'[\bot]$ is the interpolant. Indeed, if $\Pi'$ contains only structural connectives from $\mathbf{L}$ then so does the interpolant.
\end{proof}
The following result is a consequence of the proposition above and of Proposition \ref{lem:localmaeeasyr}  and Theorem \ref{thm:interpolation_LE}.
\begin{corollary}
    Lyndon interpolation property holds for fundamental logic and its basic modal expansions with tense connectives.
\end{corollary}

\section{Conclusions and further directions}
\label{sec:conclusions-interpolation}

In the present paper, we proved that the Lyndon interpolation property holds for the basic LE-logic over any LE-signature. Our proof is constructive, in the sense that interpolants can be extracted algorithmically from derivations in the corresponding display calculus. Moreover, the method is modular and extendable: we generalized the result to axiomatic extensions of LE-logics captured by augmenting the basic display calculus with suitable analytic structural rules. In particular, we identified a class of \emph{interpolation-safe rules} preserving the local Maehara-Lyndon interpolation property. To illustrate the applicability and modularity of the framework, we established Lyndon interpolation for  fundamental logic and its tense modal expansion.

A natural direction for further research concerns the extension of the present framework to \emph{inception display calculi}, recently introduced in Chapter~7 of \cite{unified-correspondence-proof-theoretically} and further developed in \cite{Inception}. As discusseed earlier on (cf.~Proposition \ref{prop:type5}), the expressivity of proper display calculi has been characterized in \cite{ucaaptt} in terms of a proper subclass of the class of {\em inductive axioms}, i.e.~the largest class of axioms to which the techniques of algorithmic correspondence theory \cite{conradie2012algorithmic, conradie2019algorithmic,CoGhPa14} can be uniformly applied.  Inception calculi generalize standard display calculi by allowing derivational nesting through higher-level assumptions, thereby providing analytic proof systems for logics axiomatized by any inductive axioms. 
Extending the present interpolation method to this setting would require a suitable generalization of the local interpolation property from ordinary analytic structural rules to inception rules. Such an extension would significantly broaden the class of axiomatic extensions for which constructive interpolation methods are available.

More generally, one of the main long-term goals of this line of research is the identification of sufficiently general syntactic conditions on analytic structural rules (and, correspondingly, on the analytic axioms they capture) guaranteeing the validity of various interpolation properties. Ideally, one would aim at a proof-theoretic interpolation meta-theorem for display calculi analogous to Belnap's celebrated meta-theorem on cut elimination. From this perspective, the class of interpolation-safe rules introduced in the present paper can be regarded as a first step toward a broader structural theory of interpolation for modular proof systems.

Towards this end, taking stock of the proof-theoretic approach developed in Chapter~6 of \cite{unified-correspondence-proof-theoretically}, which is based on the methodology introduced in \cite{Gore11}, it would be interesting to reformulate interpolation properties in terms of display-equivalent sequents. We expect that this perspective could lead to a finer-grained proof of Theorem~3.9 in the present paper.

\bibliography{ref}

\begin{thebibliography}{19}
\expandafter\ifx\csname natexlab\endcsname\relax\def\natexlab#1{#1}\fi
\providecommand{\url}[1]{\texttt{#1}}
\providecommand{\href}[2]{#2}
\providecommand{\path}[1]{#1}
\providecommand{\DOIprefix}{doi:}
\providecommand{\ArXivprefix}{arXiv:}
\providecommand{\URLprefix}{URL: }
\providecommand{\Pubmedprefix}{pmid:}
\providecommand{\doi}[1]{\href{http://dx.doi.org/#1}{\path{#1}}}
\providecommand{\Pubmed}[1]{\href{pmid:#1}{\path{#1}}}
\providecommand{\bibinfo}[2]{#2}
\ifx\xfnm\relax \def\xfnm[#1]{\unskip,\space#1}\fi
\bibitem[{D'Agostino(2008)}]{DAgostino2008}
\bibinfo{author}{G.~D'Agostino},
\newblock \bibinfo{title}{Interpolation in non-classical logics},
\newblock \bibinfo{journal}{Synthese} \bibinfo{volume}{164} (\bibinfo{year}{2008}) \bibinfo{pages}{421--435}. \URLprefix \url{http://www.jstor.org/stable/40271081}.
\bibitem[{Fussner(2026)}]{fussner2025interpolationnonclassicallogics}
\bibinfo{author}{W.~Fussner},
\newblock \bibinfo{title}{{Interpolation in Non-Classical Logics}},
\newblock in: \bibinfo{editor}{B.~ten Cate}, \bibinfo{editor}{J.~C. Jung}, \bibinfo{editor}{P.~Koopmann}, \bibinfo{editor}{C.~Wernhard}, \bibinfo{editor}{F.~Wolter} (Eds.), \bibinfo{booktitle}{{Theory and Applications of Craig Interpolation}}, \bibinfo{publisher}{Ubiquity Press}, \bibinfo{year}{2026}.
\bibitem[{van~der Giessen et~al.(2026)van~der Giessen, Jalali, and Kuznets}]{GJK2025interpolationprooftheory}
\bibinfo{author}{I.~van~der Giessen}, \bibinfo{author}{R.~Jalali}, \bibinfo{author}{R.~Kuznets},
\newblock \bibinfo{title}{{Interpolation in Proof Theory}},
\newblock in: \bibinfo{editor}{B.~ten Cate}, \bibinfo{editor}{J.~C. Jung}, \bibinfo{editor}{P.~Koopmann}, \bibinfo{editor}{C.~Wernhard}, \bibinfo{editor}{F.~Wolter} (Eds.), \bibinfo{booktitle}{{Theory and Applications of Craig Interpolation}}, \bibinfo{publisher}{Ubiquity Press}, \bibinfo{year}{2026}.
\bibitem[{Lyon et~al.(2020)Lyon, Tiu, Gor\'{e}, and Clouston}]{Lyon2020}
\bibinfo{author}{T.~Lyon}, \bibinfo{author}{A.~Tiu}, \bibinfo{author}{R.~Gor\'{e}}, \bibinfo{author}{R.~Clouston},
\newblock \bibinfo{title}{Syntactic interpolation for tense logics and bi-intuitionistic logic via nested sequents},
\newblock in: \bibinfo{editor}{M.~Fernandez}, \bibinfo{editor}{A.~Muscholl} (Eds.), \bibinfo{booktitle}{28th EACSL Annual Conference on Computer Science Logic (CSL 2020)}, \bibinfo{publisher}{Leibniz International Proceedings in Informatics}, \bibinfo{year}{2020}, pp. \bibinfo{pages}{1--16}.
\bibitem[{Kuznets(2016)}]{Kuznets16}
\bibinfo{author}{R.~Kuznets},
\newblock \bibinfo{title}{Proving {Craig} and {Lyndon} interpolation using labelled sequent calculi},
\newblock in: \bibinfo{editor}{L.~Michael}, \bibinfo{editor}{A.~Kakas} (Eds.), \bibinfo{booktitle}{Logics in Artificial Intelligence}, \bibinfo{publisher}{Springer International Publishing}, \bibinfo{address}{Cham}, \bibinfo{year}{2016}, pp. \bibinfo{pages}{320--335}.
\bibitem[{Kuznets(2015)}]{kuznets2015interpolation}
\bibinfo{author}{R.~Kuznets},
\newblock \bibinfo{title}{Interpolation method for multicomponent sequent calculi},
\newblock in: \bibinfo{booktitle}{International Symposium on Logical Foundations of Computer Science}, \bibinfo{organization}{Springer}, \bibinfo{year}{2015}, pp. \bibinfo{pages}{202--218}.
\bibitem[{Conradie and Palmigiano(2019)}]{conradie2019algorithmic}
\bibinfo{author}{W.~Conradie}, \bibinfo{author}{A.~Palmigiano},
\newblock \bibinfo{title}{Algorithmic correspondence and canonicity for non-distributive logics},
\newblock \bibinfo{journal}{Annals of Pure and Applied Logic} \bibinfo{volume}{170} (\bibinfo{year}{2019}) \bibinfo{pages}{923--974}.
\bibitem[{Conradie et~al.(2020)Conradie, Palmigiano, Robinson, and Wijnberg}]{conradie2020non}
\bibinfo{author}{W.~Conradie}, \bibinfo{author}{A.~Palmigiano}, \bibinfo{author}{C.~Robinson}, \bibinfo{author}{N.~M. Wijnberg},
\newblock \bibinfo{title}{Non-distributive logics: from semantics to meaning},
\newblock in: \bibinfo{editor}{A.~Rezus} (Ed.), \bibinfo{booktitle}{Contemporary Logic and Computing}, volume~\bibinfo{volume}{1} of \textit{\bibinfo{series}{Landscapes in Logic}}, \bibinfo{publisher}{College Publications}, \bibinfo{year}{2020}, pp. \bibinfo{pages}{38--86}.
\bibitem[{Greco et~al.(2024)Greco, Jipsen, Liang, Palmigiano, and Tzimoulis}]{LEframes}
\bibinfo{author}{G.~Greco}, \bibinfo{author}{P.~Jipsen}, \bibinfo{author}{F.~Liang}, \bibinfo{author}{A.~Palmigiano}, \bibinfo{author}{A.~Tzimoulis},
\newblock \bibinfo{title}{Algebraic proof theory for {LE}-logics},
\newblock \bibinfo{journal}{ACM Trans. Comput. Logic} \bibinfo{volume}{25} (\bibinfo{year}{2024}). \URLprefix \url{https://doi.org/10.1145/3632526}. \DOIprefix\doi{10.1145/3632526}.
\bibitem[{Almeida et~al.(2026)Almeida, Bezhanishvili, and Lemal}]{almeida2026superamalgamationmodallatticesnondistributive}
\bibinfo{author}{R.~N. Almeida}, \bibinfo{author}{N.~Bezhanishvili}, \bibinfo{author}{S.~Lemal}, \bibinfo{title}{Superamalgamation for modal lattices via non-distributive dualities}, \bibinfo{year}{2026}. \URLprefix \url{https://arxiv.org/abs/2602.20380}. \href{http://arxiv.org/abs/2602.20380}{{\tt arXiv:2602.20380}}.
\bibitem[{Holliday(2023)}]{Holliday2023}
\bibinfo{author}{W.~Holliday},
\newblock \bibinfo{title}{A fundamental non-classical logic},
\newblock \bibinfo{journal}{Logics} \bibinfo{volume}{1} (\bibinfo{year}{2023}) \bibinfo{pages}{36--79}.
\bibitem[{Holliday(2024)}]{Holliday2024}
\bibinfo{author}{W.~H. Holliday},
\newblock \bibinfo{title}{Modal logic, fundamentally},
\newblock in: \bibinfo{editor}{A.~Ciabattoni}, \bibinfo{editor}{D.~Gabelaia}, \bibinfo{editor}{I.~Sedl\'{a}r} (Eds.), \bibinfo{booktitle}{Advances in Modal Logic, Vol. 15}, \bibinfo{publisher}{College Publications}, \bibinfo{year}{2024}.
\bibitem[{Greco et~al.(2018)Greco, Ma, Palmigiano, Tzimoulis, and Zhao}]{ucaaptt}
\bibinfo{author}{G.~Greco}, \bibinfo{author}{M.~Ma}, \bibinfo{author}{A.~Palmigiano}, \bibinfo{author}{A.~Tzimoulis}, \bibinfo{author}{Z.~Zhao},
\newblock \bibinfo{title}{Unified correspondence as a proof-theoretic tool},
\newblock \bibinfo{journal}{Journal of Logic and Computation} \bibinfo{volume}{28} (\bibinfo{year}{2018}) \bibinfo{pages}{1367--1442}.
\bibitem[{Chen et~al.(2022)Chen, Greco, Palmigiano, and Tzimoulis}]{ChnGrePalTzi21}
\bibinfo{author}{J.~Chen}, \bibinfo{author}{G.~Greco}, \bibinfo{author}{A.~Palmigiano}, \bibinfo{author}{A.~Tzimoulis},
\newblock \bibinfo{title}{Syntactic completeness of proper display calculi},
\newblock \bibinfo{journal}{ACM Transactions on Computational Logic} \bibinfo{volume}{23:4} (\bibinfo{year}{2022}) \bibinfo{pages}{1--46}.
\bibitem[{{De Domenico}(2025)}]{unified-correspondence-proof-theoretically}
\bibinfo{author}{A.~{De Domenico}}, \bibinfo{title}{Unified correspondence, proof-theoretically}, \bibinfo{type}{Phd-thesis - research and graduation internal}, Vrije Universiteit Amsterdam, \bibinfo{year}{2025}. \DOIprefix\doi{10.5463/thesis.1412}.
\bibitem[{De~Domenico et~al.(2026)De~Domenico, Greco, and Palmigiano}]{Inception}
\bibinfo{author}{A.~De~Domenico}, \bibinfo{author}{G.~Greco}, \bibinfo{author}{A.~Palmigiano}, \bibinfo{title}{Inception display calculi}, \bibinfo{year}{2026}. \bibinfo{note}{To appear in the Special Issue of \emph{Studia Logica} in memory of Nuel Belnap (1930--2024), edited by Thomas M{\"u}ller, Tomasz Placek, and Heinrich Wansing}.
\bibitem[{Conradie and Palmigiano(2012)}]{conradie2012algorithmic}
\bibinfo{author}{W.~Conradie}, \bibinfo{author}{A.~Palmigiano},
\newblock \bibinfo{title}{Algorithmic correspondence and canonicity for distributive modal logic},
\newblock \bibinfo{journal}{Annals of Pure and Applied Logic} \bibinfo{volume}{163} (\bibinfo{year}{2012}) \bibinfo{pages}{338--376}.
\bibitem[{Conradie et~al.(2014)Conradie, Ghilardi, and Palmigiano}]{CoGhPa14}
\bibinfo{author}{W.~Conradie}, \bibinfo{author}{S.~Ghilardi}, \bibinfo{author}{A.~Palmigiano},
\newblock \bibinfo{title}{{Unified Correspondence}},
\newblock in: \bibinfo{editor}{A.~Baltag}, \bibinfo{editor}{A.~Smets} (Eds.), \bibinfo{booktitle}{{J}ohan van {B}enthem on Logic and Information Dynamics}, volume~\bibinfo{volume}{5} of \textit{\bibinfo{series}{Outstanding Contributions to Logic}}, \bibinfo{publisher}{Springer International Publishing}, \bibinfo{year}{2014}, pp. \bibinfo{pages}{933--975}.
\bibitem[{Brotherston and Gor{\'e}(2011)}]{Gore11}
\bibinfo{author}{J.~Brotherston}, \bibinfo{author}{R.~Gor{\'e}},
\newblock \bibinfo{title}{{Craig Interpolation in Displayable Logics}},
\newblock in: \bibinfo{editor}{K.~Br{\"u}nnler}, \bibinfo{editor}{G.~Metcalfe} (Eds.), \bibinfo{booktitle}{Automated Reasoning with Analytic Tableaux and Related Methods}, \bibinfo{publisher}{Springer Berlin Heidelberg}, \bibinfo{address}{Berlin, Heidelberg}, \bibinfo{year}{2011}, pp. \bibinfo{pages}{88--103}.

\end{thebibliography}

\appendix

\section*{Declaration on Generative AI}
  The author(s) have not employed any Generative AI tools.
  \newline
\end{document}